\documentclass{amsart}

\usepackage{amsmath, amssymb,graphicx,mathrsfs,enumerate}
\usepackage[all]{xy}

\usepackage{latexsym}

\newtheorem{proposition}{Proposition}[section]
\newtheorem{lemma}[proposition]{Lemma}
\newtheorem{corollary}[proposition]{Corollary}
\newtheorem{theorem}[proposition]{Theorem}

\newcommand{\diamd}{\overrightarrow{\mathrm{diam}}}
\DeclareMathOperator{\diam}{diam}
\DeclareMathOperator{\GL}{GL}

\begin{document}

\title{Bounds for the diameters of orbital graphs of affine groups}

\author{Attila Mar\'oti}
\address{Alfr\'ed R\'enyi Institute of Mathematics, Re\'altanoda utca 13-15, H-1053, Budapest, Hungary}
\email{maroti@renyi.hu}

\author{Saveliy V. Skresanov}
\address{Sobolev Institute of Mathematics, 4 Acad. Koptyug avenue, Novosibirsk, Russia}
\email{skresan@math.nsc.ru}

\thanks{The project leading to this application has received funding from the European Research Council (ERC) under the European Union's Horizon 2020 research and innovation programme (grant agreement No 741420). The first author was supported by the National Research, Development and Innovation Office (NKFIH) Grant No.~K138596, No.~K132951 and Grant No.~K138828.
The second author was supported by RAS Fundamental Research Program, project FWNF-2022-0002.}

\begin{abstract}
	General bounds are presented for the diameters of orbital graphs of finite affine primitive permutation groups.
	For example, it is proved that the orbital diameter of a finite affine primitive permutation group with a nontrivial point
	stabilizer $ H \leq \GL(V) $, where the vector space $ V $ has dimension $ d $ over the prime field,
	can be bounded in terms of $ d $ and $ \log |V| / \log |H| $ only.
	Several infinite families of affine primitive permutation groups with large orbital diameter are constructed.
	The results are independent from the classification of finite simple groups.   
\end{abstract}
\dedicatory{Dedicated to Pham Huu Tiep on the occasion of his 60th birthday}
\maketitle

\section{Introduction}

Connections between finite permutation groups and graphs quite often allow one to express
some abstract group-theoretic properties in combinatorial terms. One of the classical approaches
is the study of orbital graphs of permutation groups. Recall that if $ G $ is
a permutation group acting on a finite set $ X $, then an \emph{orbital graph} of $ G $
is a graph with vertex set $ X $ whose arc set is an orbit of $ G $ on $ X \times X $.
An orbital graph whose arcs are a subset of the diagonal $ \{ (x, x) \mid x \in X \} $ is
called a \emph{diagonal} orbital graph.

A transitive permutation group is called \emph{primitive} if it has no
nontrivial proper block of imprimitivity or, equivalently, if a point
stabilizer is a maximal subgroup in the group. Primitive permutation
groups can also be characterized in terms of their orbital graphs.
By a criterion of Higman~\cite{Higman}, a finite permutation group is primitive if and only if all its
non-diagonal orbital graphs are connected. The degree of an orbital graph of a finite
primitive permutation group is also of fundamental importance. By Sims conjecture,
resolved by Cameron, Praeger, Saxl and Seitz~\cite{CPSS} (see also a simplified proof by Pyber and Tracey~\cite{PT})
if some non-diagonal orbital graph of a finite primitive permutation group
has degree bounded by $ d $, then a point stabilizer of the group has size bounded in terms of $ d $. 

The subject of this paper is the diameter of an orbital graph of a finite primitive permutation group, or \emph{orbital diameter}.
Using model theory, Liebeck, Macpherson and Tent~\cite{LMT} described
finite primitive permutation groups whose non-diagonal orbital graphs have bounded diameter
(we note that in~\cite{LMT} orbital graphs are considered to be undirected),
see also the papers of Sheikh~\cite{Sh} and Rekv\'enyi~\cite{Re}.
The converse problem of finding upper bounds on the diameters of orbital graphs of primitive groups
was also considered in~\cite{LMT}, but in the case of affine groups only a partial result was obtained, see~\cite[Lemma~3.1]{LMT}.

Recall that a finite primitive permutation group $ G $ acting on a set $ X $ is called \emph{affine}
if it contains a unique minimal normal subgroup $V$ which is elementary abelian (and regular as a permutation group).
In this case a stabilizer $H$ in $G$ of a point in $X$ acts linearly on $V$,
viewed as a vector space over the prime field $\mathbb{F}_{p}$, where $p$ is prime.
Moreover $V$ is a faithful and irreducible $\mathbb{F}_{p}H$-module and conversely,
if $ V $ is a faithful and irreducible $ \mathbb{F}_pH $-module for a finite group $ H $,
then the semidirect product $ HV $ may be viewed as an affine primitive permutation group with socle $ V $.
The main goal of this paper is to provide detailed upper bounds on the diameters of orbital graphs
of affine primitive permutation groups, and complement the results of~\cite{LMT}.

Another motivation is a problem of finding a $ n^{o(1)} $ upper bound on diameters of orbital graphs
of some classes of primitive permutation groups of degree $ n $ and, more generally, a bound on diameters of constituent
graphs of primitive coherent configurations. See Pyber's paper~\cite{Pyb} proving a logarithmic bound for distance-regular graphs
and his conjecture in~\cite[p.~119]{LL}. We show that there exist several infinite families of affine primitive permutation groups
of degree $ n = p^d $, where $ p $ is a prime, with orbital diameter at least $ (p-1)d/4 $ (see Section~\ref{secExamples}).
In particular, the $ n^{o(1)} $ bound is not achievable for affine groups. We note that a weaker
version of this was found earlier by Pyber (personal communication).

Our third motivation comes from additive combinatorics. 
Cochrane and Cipra \cite[Theorem~1.2]{CC} proved the following result on the Waring problem in finite fields.
Let $ A $ be a nontrivial multiplicative subgroup of $ F^\times $, where $ F $ is a finite
field, and assume that $ A $ generates $ F $ additively.  Then $ n \cdot A = F $
for every $ n \geq 633 \cdot |F|^{\frac{\log 4}{\log |A|}} $, where $ n \cdot A $ denotes the sum of $ n $ copies of the set
$ A $ (see Proposition~\ref{CC} for a more precise statement), in particular, the diameter of a non-diagonal
orbital graph of the affine primitive permutation group $ AF $ is bounded by $ n $.
Another goal of this paper is to present a generalization of this theorem (see Theorem~\ref{main}).

The theorem of Cochrane and Cipra should be compared to the following results of Babai.
Let $r$ and $p$ be primes such that $r$ divides $p-1$ and let $H(p,r)$ be the unique nonabelian group of order~$pr$.
This group can be generated by two elements. Let $\diam_{\min}(H(p,r))$ and $\diam_{\max}(H(p,r))$ denote the minimum and, respectively,
maximum diameter of a Cayley graph of $H(p,r)$ over all possible pairs of generators.
It was shown by Babai in an unpublished work that if $r$ is bounded and $p \to \infty$ then both
$\diam_{\min}(H(p,r))$ and $\diam_{\max}(H(p,r))$ have the order of $\Theta(p^{1/(r-1)})$ (see \cite[Theorem 6.1]{many}).
If $r > p^{1/2 + c}$ for some constant $c > 0$ then $\diam_{\min}(H(p,r)) = \Theta(r^{1/2})$ and
$\diam_{\max}(H(p,r)) = \Theta(r)$ (see \cite[Theorem 6.2]{many}). The second result is connected with
the Waring problem in $ \mathbb{F}_p $, see~\cite[Section~6]{many} for details.

Another related result is a theorem of Peluse~\cite{P} on exponential sums over orbits of a linear group.
It was shown that for every positive integer $ d $ and positive numbers $ \delta $, $ \beta $ there is a positive
$ \epsilon $ such that whenever $ HV $ is an affine primitive permutation group where
$ H \leq \GL(V) $ and $ V $ is a vector space of dimension $ d $ over $ \mathbb{F}_p $ and $ \Delta $ is
a nonzero orbit of $ H $ on $ V $ with (1)~$ |\Delta| \geq p^\delta $ and (2)~$ |\Delta \cap P| \leq |\Delta|^{1 - \beta} $ for every
hyperplane $ P $ in $ V $, then the absolute value of an exponential sum over $ \Delta $ is bounded by $ p^{-\epsilon}|\Delta| $.
By Fourier analysis, this bound implies that the maximum diameter of a non-diagonal orbital graph of the affine
primitive permutation group $HV$ is bounded in terms of $ d $, $ \delta $ and $ \beta $ only.
The third goal of the paper is to show that the maximum diameter of a non-diagonal orbital graph
of an affine primitive permutation group is bounded in terms of $ d $ and $ \delta $ provided that (1) holds (see Corollary~\ref{ratioBounds}).

To state our main result, let $ V $ be a faithful irreducible $ \mathbb{F}_pH $-module for a finite group $ H $,
and let $ d $ be the dimension of $ V $ over $ \mathbb{F}_p $.
Denote by $ \diamd(V, H) $ the supremum of the diameters of non-diagonal
orbital graphs of $ HV $ where graphs are considered to be oriented, and let $ \diam(V, H) $ be the
supremum of diameters of non-diagonal orbital graphs where we forget about arc directions and consider graphs to be undirected.
It follows from elementary diameter estimates (see Proposition~\ref{lowerBounds} and~\ref{centerIneq}) that
$$ \frac{1}{2}(|V|^{1/|H|}-1) \leq \diam(V, H) \leq \diamd(V, H) \leq (p-1)d. $$

Our main result may be considered an improvement of the previous upper bound on the diameter in the case when $ p $ is large. 

\begin{theorem}\label{main}
	Let $ HV \leq \mathrm{AGL}(V) $ be an affine primitive permutation group with $ V $ a vector space
	of dimension $d$ over $\mathbb{F}_{p}$, $ p $ prime, and $H$ a point stabilizer.
	\begin{enumerate}
		\item If $H$ has a composition factor isomorphic to a finite simple group of Lie type in characteristic $p$, then
		$$\mathrm{diam}(V,H) < 2^{22d^3},$$
		in particular, the diameter is bounded in terms of $ \dim V $ only.
		\item There exists a function $ J(d) $ depending only on $ d $ such that whenever
			$H$ does not have a composition factor isomorphic to a finite simple group of Lie type in characteristic $p$,
			and $ |H| \geq J(d)^2 $, then
			$$\overrightarrow{\mathrm{diam}}(V,H) <  2^{18d^2} \cdot {|V|}^{\frac{d\log 64}{\log |H|}}.$$	
	\end{enumerate}
\end{theorem}

Part~(2) of Theorem \ref{main} may be viewed as a generalization of the
aforementioned theorem of Cochrane and Cipra~\cite{CC} for arbitrary affine primitive groups
in the case when $ d $ is bounded.

The function $ J(d) $ from the statement of (2) of Theorem~\ref{main}
is precisely the function from the Larsen and Pink theorem, see Proposition~\ref{LP},
which is a classification-free version of Weisfeiler's theorem on the structure of linear groups~\cite{W, Wnotes}.
We will show in Section~\ref{secExamples} that there exist groups $ H $ of
size bounded in terms of $ d $ with orbital graphs of large diameter,
so the condition that $ |H| $ is larger than some constant depending on
$ d $ is essential in~(2) of Theorem~\ref{main}.

In~\cite[Part (1) of Theorem~1.1]{LMT} it was shown that if a class of affine primitive permutation
groups has bounded orbital diameter, then this class consists of the so-called groups of $ t $-bounded
classical type, for some bounded $ t $. It follows from the definition of groups of $ t $-bounded classical
type that an affine group $ HV $ with $ d = \dim V $ is of $ t $-bounded classical type for all $ t \geq d $,
and therefore the conclusion of~\cite[Theorem~1.1]{LMT} is nontrivial only for classes of groups with unbounded $ d $
(see the remark before~\cite[Lemma~3.2]{LMT}).

As a corollary to Theorem~\ref{main}, one can show that diameters are
controlled by the ratio $ \log |V| / \log |H| $, when $ d $ is bounded.

\begin{corollary}\label{ratioBounds}
	For any $ d $ there exists a constant $ f(d) $ depending on $ d $ only such that
	for any affine primitive permutation group $ HV \leq \mathrm{AGL}(V) $ with
	$V$ a vector space of dimension $d$ over $\mathbb{F}_{p}$ and $ 1 < H \leq \mathrm{GL}(V)$, we have
	$$ \frac{\log |V|}{3\log |H|} \leq \diam(V,H) \leq f(d)^{\frac{\log |V|}{\log |H|}}. $$
	In particular, when $ d $ is bounded, $ \diam(V, H) $ is bounded if and only if
	$ \frac{\log|V|}{\log |H|} $ is bounded.
\end{corollary}

In Section~\ref{secExamples} we will show that $ f(d) $ grows at least linearly.

In order to prove Theorem~\ref{main} we establish the following technical result
providing a more precise upper bound on the diameters of orbital graphs.
Recall that a finite group is called a $ p' $-group if its order is not divisible by a prime $ p $.

\begin{theorem}\label{mainExplicit}
	Let $HV \leq \mathrm{AGL}(V)$ be an affine primitive permutation group where $V$ is a vector space of dimension $d$ over $\mathbb{F}_{p}$
	and $H \leq \mathrm{GL}(V)$ acts irreducibly on~$V$. Let $A$ be a nontrivial abelian $p'$-subgroup of $H$ and let $k$ be the number of
	irreducible summands of the completely reducible $\mathbb{F}_{p}A$-module $V$. Then
	$$\mathrm{diam}(V,H) < 322 d \cdot 144^{k(k+1)} \cdot {|V|}^{\frac{k (k+1)\log 4}{\log |A|}}.$$
	Moreover, if $A$ is normal in $H$, then
	$$\overrightarrow{\mathrm{diam}}(V,H) <  d \cdot 2576^{k(k+1)} \cdot {|V|}^{\frac{(k+1)\log 4}{\log |A|}}.$$	
\end{theorem}

The proof of Theorem \ref{mainExplicit} relies on the result of Cochrane and Cipra~\cite[Theorem~1.2]{CC}.

It should be mentioned that our results are free from the classification of finite simple groups.
By utilizing the classification it is possible to give a good bound on the function $ J(d) $
(see, for example, the bounds in \cite{Wnotes} or \cite{Collins}) and improve the upper bounds in the case
when the group contains a composition factor isomorphic to a finite simple
group of Lie type in characteristic $ p $, but these questions are out of scope of this paper.
	
The structure of the paper is as follows. In Section~\ref{secPrelim} we provide
necessary preliminaries on additive properties of subsets and reformulate the
problem of bounding diameters in those terms.  In Section~\ref{secElem} we give
elementary upper and lower bounds on the diameters, in particular, we provide an
upper bound in terms of the intersection of the point stabilizer with the group of
scalar matrices.  Section~\ref{secExplicit} contains the proof of Theorem
\ref{mainExplicit}. In Section~\ref{secLP} we derive Theorem~\ref{main} from Theorem~\ref{mainExplicit},
and in Section~\ref{secCor} we derive Corollary~\ref{ratioBounds}.
In Section~\ref{secExamples} we provide examples of affine primitive permutation groups with large diameter,
proving that some of our estimates are tight.

\section{Preliminaries}\label{secPrelim}

Let $ V $ be a finite abelian group, and let $ \Delta, \Gamma $ be subsets of $ V $.
We define the sum, difference and negation of subsets as usual:
\begin{align*}
	\Delta + \Gamma &= \{ x + y \mid x \in \Delta, y \in \Gamma \},\\
	\Delta - \Gamma &= \{ x - y \mid x \in \Delta, y \in \Gamma \},\\
	- \Delta &= \{ -x \mid x \in \Delta \}.
\end{align*}
For an integer $ n \geq 1 $ let $ n \cdot \Delta $ denote the sum of $ n $ copies of $ \Delta $.
The following proposition records several properties of subset sums, which will be used without further notice.

\begin{proposition}
	Let $ V $ be a finite abelian group, and let $ \Delta \subseteq V $.
	\begin{enumerate}
		\item If $ m \leq n $, then $ |m \cdot \Delta| \leq |n \cdot \Delta| $.
			Furthermore, if $ 0 \in \Delta $, then $ m \cdot \Delta \subseteq n \cdot \Delta $.
		\item If $ m \cdot \Delta = V $, then for any $ n \geq m $ we have $ n \cdot \Delta = V $.
		\item If $ \Delta + \Delta \subseteq \Delta $ and $ \Delta $ is nonempty, then $ \Delta $
			is a subgroup of $ V $.
	\end{enumerate}
\end{proposition}
\begin{proof}
	If $ \Delta $ is nonempty and $ x \in \Delta $, then for any $ \Gamma \subseteq V $ we have
	$$ |\Gamma| = |\Gamma + x| \leq |\Gamma + \Delta|, $$
	where the last inequality holds since $ \Gamma + x \subseteq \Gamma + \Delta $.
	Furthermore, if $ 0 \in \Delta $, then $ \Gamma \subseteq \Gamma + \Delta $, and~(1) follows.

	Property~(2) is a direct consequence of~(1), and~(3) is well-known for arbitrary finite groups.
\end{proof}

Let $G$ be a permutation group acting on a finite set $X$. An orbital graph for $(X,G)$ is a graph with vertex set $X$
whose arc set is an orbit of $G$ on $X \times X$; in general, this is a directed graph. The orbital graph with edge set
$\{ (x,x) : x \in X \}$ is called the diagonal orbital graph. The criterion of Higman~\cite{Higman} states that
a transitive permutation group $G$ acting on $X$ is primitive if and only if all non-diagonal orbital graphs are (strongly) connected,
see~\cite[Theorems~1.9 and~1.10]{Cameron} for the proof.

Assume that $G$ is an affine primitive permutation group with socle $V$.
The group $V$ is elementary abelian and acts regularly on $X$, so it can be identified with $X$ in a natural way.
Viewing $V$ as a vector space of dimension $d$ over the finite field $\mathbb{F}_{p}$
the group $G$ can be considered a subgroup of the affine general linear group $\mathrm{AGL}(V)$.
Therefore $G$ decomposes as a semidirect product $HV$ where $H$ is the stabilizer of $0 \in V$,
and $V$ is a faithful irreducible $\mathbb{F}_pH$-module.

In the case of an affine permutation group one can easily see that its orbital graphs are Cayley graphs.
Indeed, if $ \Delta $ is an orbit of $ H $ on $ V $, then the corresponding orbital graph
has arc set $ \{ (x, y) \in V \times V \mid x-y \in \Delta \} $, so $ \Delta $ is the connection set.
Notice that we obtain the diagonal orbital graph when $ \Delta $ is the zero orbit.

If $ x_1, \dots, x_k \in V $ is a directed path in an orbital graph corresponding to the orbit $ \Delta $
of $ H $ on $ V $, then $ x_1 - x_2 \in \Delta, \; x_2 - x_3 \in \Delta, \dots, x_{k-1} - x_k \in \Delta $,
and therefore $ x_1 - x_k \in (k-1) \cdot \Delta $. It follows that the (directed) diameter of the corresponding
orbital graph is equal to the minimal number $ n \geq 1 $ such that
$$ \{ 0 \} \cup (1 \cdot \Delta) \cup (2 \cdot \Delta) \cup \dots \cup (n \cdot \Delta) = V, $$
or, in other words, to the minimal $ n \geq 1 $ such that $ n \cdot (\Delta \cup \{ 0 \}) = V $.
Observe that such $ n $ always exists for a nonzero orbit $ \Delta $, since $ H $ acts irreducibly
on $ V $ and therefore $ \Delta $ spans $ V $ over $ \mathbb{F}_p $.

Denote by $ \diamd(V, H) $ the supremum of the diameters of non-diagonal orbital graphs of $ (X, G) $.
Let $ \Delta_1, \dots, \Delta_r $ be all the nonzero orbits of $ H $ on $ V $. Then
\begin{equation}\label{eqDiamd}
	\diamd(V, H) = \min \{ n \in \mathbb{N} \mid n \cdot (\Delta_i \cup \{ 0 \}) = V \text{ for all } i = 1, \dots, r \}.
\end{equation}
If we forget about the arc direction of an orbital graph, we can consider its undirected diameter; let
$ \diam(V, H) $ denote the supremum of the undirected diameters of non-diagonal orbital graphs of $ (X, G) $.
If $ \Delta $ is the connection set of some orbital graph, then the corresponding undirected graph has
connection set $ \Delta \cup -\Delta $, in particular, we obtain the formula
\begin{equation}\label{eqDiam}
	\diam(V, H) = \min \{ n \in \mathbb{N} \mid n \cdot (\Delta_i \cup -\Delta_i \cup \{ 0 \}) = V \text{ for all } i = 1, \dots, r \}.
\end{equation}
Clearly the orbits of the group $ H\langle -1 \rangle $ on $ V $ are $ \Delta_i \cup -\Delta_i $,
$ i = 1, \dots, r $, so $ \diam(V, H) = \diamd(V, H\langle -1 \rangle) $.

The following lemma shows that if all elements of some nontrivial subspace lie at distance at most $ m $ from $ 0 $
in an orbital graph, then the whole orbital graph has diameter at most $ dm $, where $ d $ is the dimension of $ V $.
\begin{lemma}\label{subgroupTrick}
	For any subgroup $A$ of $H$, any nontrivial subspace $W$ of $V$ and any nonzero vector $v$ in $V$,
	if $ W \subseteq m \cdot (v^A \cup \{ 0 \}) $ for some $ m $, then $ dm \cdot (v^H \cup \{ 0 \}) = V $.
\end{lemma}
\begin{proof}
	Let $u \in W $ be a nonzero vector. The orbit $ u^H $ spans $ V $ over $ \mathbb{F}_{p} $,
	so there exist elements $ h_1, \dots, h_d \in H $ such that $ u^{h_1}, \dots, u^{h_d} $
	is a basis of $ V $ over $ \mathbb{F}_{p} $. Since $ \mathbb{F}_p u \subseteq W \subseteq m \cdot (v^A \cup \{ 0 \}) $
	and $ A \leq H $, we have
	$$ \mathbb{F}_{p}u^{h_1}, \dots, \mathbb{F}_{p}u^{h_d} \subseteq m \cdot (v^{H} \cup \{ 0 \}). $$
	Therefore $ V = \mathbb{F}_{p}u^{h_1} + \dots + \mathbb{F}_{p}u^{h_d} \subseteq dm \cdot (v^{H} \cup \{ 0 \}) $, and the claim is proved.
\end{proof}

By considering groups $ H\langle -1 \rangle $ and $ A\langle -1 \rangle $ we obtain an undirected
version of the above result: if $ W \subseteq m \cdot (v^A \cup -v^{A} \cup \{ 0 \}) $ for some $ m $,
then $ dm \cdot (v^H \cup -v^H \cup \{ 0 \}) = V $.

\section{Elementary diameter estimates}\label{secElem}

Let $HV \leq \mathrm{AGL}(V)$ be an affine primitive permutation group where $V$ is a vector space of dimension $d$
over the prime field $\mathbb{F}_{p}$ and $H \leq \mathrm{GL}(V)$. 
First we prove two general lower bounds for the diameter. We follow the proof of~\cite[Proposition~1.1]{BS}
for our second lower bound, see also~\cite[Theorem~3]{AB}.

\begin{proposition}\label{lowerBounds}
	For $ H > 1 $ let $ s $ be the size of the smallest nonzero orbit of $ H $ on $ V $. Then
	$$ \frac{\log |V|}{3\log |H|} \leq \frac{\log |V|}{\log (2s+1)} \leq \mathrm{diam}(V, H) $$
	and
	$$ \frac{1}{2}(|V|^{1/|H|}-1) \leq \frac{1}{2}(|V|^{1/s}-1) \leq \mathrm{diam}(V, H). $$
\end{proposition}
\begin{proof}
	Let $ \Delta $ be a nonzero orbit of $ H $ on $ V $ of size $ s $.
	Set $ n = \diam(V, H) $ and recall that $ n \cdot (\Delta \cup -\Delta \cup \{ 0 \}) = V $.
	Therefore
	$$ |V| = |n \cdot (\Delta \cup -\Delta \cup \{ 0 \})| \leq |\Delta \cup -\Delta \cup \{ 0 \}|^n \leq (2s+1)^n, $$
	and the first displayed inequalities are proved.

	Set $ \Delta = \{ x_1, \dots, x_s \} $. Then every vector $ x \in V $ can be expressed as
	the sum $ x = k_1 x_1 + \dots + k_s x_s $, where $ |k_i| \leq n $. Therefore $ |V| \leq (2n+1)^s $,
	and the second result follows.
\end{proof}

Since $ \diam(V, H) \leq \diamd(V, H) $, the provided bounds are also lower bounds for the directed diameter.

Let $ \mathbb{F}_p^\times $ denote the multiplicative group of the finite field $ \mathbb{F}_p $,
and recall that we can identify $ \mathbb{F}_p^\times $ with the center of $ \GL(V) $ in a natural way.
The next result shows that the diameter is controlled by the intersection of the group $ H $ with $ \mathbb{F}_p^\times $,
essentially generalizing~\cite[Part~(i) of Lemma~3.1]{LMT}.

\begin{proposition}\label{centerIneq}
	We have
	$$\diamd(V,H) \leq \diamd(\mathbb{F}_{p},\mathbb{F}_{p}^{\times} \cap H) \cdot d \leq (p-1)d$$
	for all $ p $, and
	$$\diam(V,H) \leq \diam(\mathbb{F}_{p},\mathbb{F}_{p}^{\times} \cap H\langle -1 \rangle) \cdot d \leq (p-1)d/2,$$
	when $ p $ is odd.
\end{proposition} 
\begin{proof}
	Set $ A = \mathbb{F}_p^\times \cap H $ and let $ v \in V $ be a nonzero vector.
	Since $ A $ is a nonzero orbit of $ A $ acting on $ \mathbb{F}_p $, we have
	$ n \cdot (A \cup \{ 0 \}) = \mathbb{F}_p $, where $ n = \diamd(\mathbb{F}_p, A) $.
	Therefore $ n \cdot (v^A \cup \{ 0 \}) $ is a nontrivial subgroup of $ V $,
	and by Lemma~\ref{subgroupTrick}, we have $ dn \cdot (v^H \cup \{ 0 \}) = V $.
	Thus $ \diamd(V, H) \leq \diamd(\mathbb{F}_p, A) \cdot d $.
	Now, $ \diamd(\mathbb{F}_p, A) \leq \diamd(\mathbb{F}_p, \langle 1 \rangle) $. 
	Since non-diagonal orbital graphs of the trivial group acting on $ \mathbb{F}_p $ are directed cycles
	of length $ p $, we obtain $ \diamd(\mathbb{F}_p, \langle 1 \rangle) = p-1 $
	and the first displayed inequalities are proved.

	To prove the second part, recall that $ \diam(V, H) = \diamd(V, H\langle -1 \rangle) $,
	therefore $ \diam(V, H) \leq \diam(\mathbb{F}_p, \mathbb{F}_p^\times \cap H\langle -1 \rangle) \cdot d $ by the previous paragraph.
	Clearly $ \diam(\mathbb{F}_p, \mathbb{F}_p^\times \cap H\langle -1 \rangle) \leq \diam(\mathbb{F}_p, \langle -1 \rangle) $.
	Finally, we have $ \diam(\mathbb{F}_p, \langle -1 \rangle) = (p-1)/2 $ for odd $ p $, since orbital graphs of
	$ \langle -1 \rangle $ acting on $ \mathbb{F}_p $ are undirected cycles of length~$ p $.
\end{proof}

In Section~\ref{secExamples} we will show that for all $ d $ and all odd $ p $
the inequalities on undirected diameter presented in Proposition~\ref{centerIneq} are sharp. Notice
that the bound on undirected diameter does not apply when $ p = 2 $, as $ \diam(\mathbb{F}_2, \langle 1 \rangle) = 1 $.

\section{Proof of Theorem~\ref{mainExplicit}}\label{secExplicit}

In this section we prove Theorem~\ref{mainExplicit} which is the main technical result
required for the proof of Theorem~\ref{main}.

Let $ H $ be an irreducible subgroup of $ \GL(V) $, where $ V $ has dimension $ d $
over the prime field $ \mathbb{F}_p $, and let $A$ be a nontrivial abelian $p'$-subgroup of $H$.
The vector space $V$ is a completely reducible $\mathbb{F}_{p}A$-module by Maschke's theorem.
We can write $ V = V_1 \oplus \dots \oplus V_k $, $ k \leq d $, where the $ V_i $ are irreducible $ \mathbb{F}_{p}A $-modules.
Write $ |V_i| = p^{f_i} $, $ i = 1, \dots, k $, for some integers $ f_i $.
Let $ A_i \leq \mathrm{GL}(V_i) $ be the group induced on $ V_i $ by~$ A $.
Since $ A $ is abelian, it induces a multiplicative group of a finite field
on each $ V_i $, so we may assume that $ A_i \leq \mathbb{F}^\times_{p^{f_i}} $ for $ i = 1, \dots, k $.

We have $ |A_i| \geq |A|^{1/k} $ for some $ i $.
Without loss of generality we may assume that for some $ j \geq 1 $ we have $ |A_i| \geq |A|^{1/k} $ for $ i = 1, \dots, j $,
and $ |A_i| < |A|^{1/k} $ for $ i = j+1, \dots, k $.
Nonzero orbits of $ A_i $ on $ V_i $ have equal sizes, and in particular, for $ i = 1, \dots, j $ nonzero orbits
of $ A_i $ on $ V_i $ have size at least $ |A|^{1/k} $, while for $ i = j+1, \dots, k $ orbit sizes are less than $ |A|^{1/k} $.

If the subgroup $ A $ is normal in $ H $, then all $ \mathbb{F}_{p}A $-modules $ V_i $, $ i = 1, \dots, k $, are isomorphic,
in particular, $ j = k $ in this case.

Every vector $ v \in V $ can be uniquely written as $ v = v_1 + \dots + v_k $, where $ v_i \in V_i $, $ i = 1, \dots, k $.
We say that $ v_i $ is the projection of $ v $ on $ V_i $.

\begin{lemma}\label{removeSmall}
	Let $ v $ be a nonzero vector from $ V $ and set $\Delta = v^{H} \cup  -v^{H} \cup \{ 0 \}$.
	If $ s \leq k-j $, then there exists $ w \in 2^s \cdot \Delta $ such that
	$ w \in V_1 \oplus \dots \oplus V_{k-s} $ and $ w $ has a nonzero projection on $ V_1 $.
\end{lemma}
\begin{proof}
	We will prove the statement by induction on $ s $.
	Suppose that $ s = 0 $. Since $ \Delta $ spans $ V $, it cannot lie inside the proper subspace $ V_2 \oplus \dots \oplus V_k $.
	Therefore there exists a vector $ w \in 1 \cdot \Delta $ such that $ w $ has a nonzero projection on $ V_1 $.
	
	Suppose that $ s > 0 $, so, in particular, $ A $ is not normal in $ H $.
	By the inductive hypothesis, there exists some vector $ u \in 2^{s-1} \cdot \Delta $ lying in
	$ V_1 \oplus \dots \oplus V_{k-s+1} $ and having a nonzero projection on $ V_1 $, that is,
	$ u = u_1 + \dots + u_{k-s+1} $ where $ u_i \in V_i $ for every $ i = 1, \dots, k-s+1 $, and $ u_1 \neq 0 $.
	Recall that the length of the $ A $-orbit of $ u_{k-s+1} $ is strictly smaller than $ |A|^{1/k} $,
	in particular, it is smaller than the length of the $ A $-orbit of $ u_1 $. Therefore there exists some $ a \in A $
	with $ u_1^a \neq u_1 $ and $ u_{k-s+1}^a = u_{k-s+1} $.
	The vector $ w = u - u^a $ has a nonzero projection on $ V_1 $ and lies in $ V_1 \oplus \dots \oplus V_{k-s} $.
	It is left to notice that $ w \in 2^{s-1} \cdot \Delta - 2^{s-1} \cdot \Delta = 2^s \cdot \Delta $, 
	where the last equality follows from $ -\Delta = \Delta $. The inductive argument is over.
\end{proof}

We need the following result of Cochrane and Cipra. 	

\begin{proposition}[{\cite[Theorem~1.2]{CC}}]
	\label{CC}
	Let $ M $ be a subgroup of the multiplicative group $ \mathbb{F}_{q}^\times $ of the finite field $ \mathbb{F}_q $.
	If $ M $ generates $ \mathbb{F}_{q} $ additively and $ |M| > 1 $, then we have
	$ n \cdot M = \mathbb{F}_{q} $ for every $ n \geq 633 (2(q-1)/|M|)^{\frac{\log 4}{\log |M|}} $.
\end{proposition}

Recall that $ A $ induces a group $ A_i $ on each $ V_i $, $ i = 1, \dots, j $,
and $ |A_i| \geq |A|^{1/k} $. We have $ n_{i} \cdot A_i = \mathbb{F}_{p^{f_i}} $ for all $ i = 1, \dots, j $
provided that $n_{i} \geq 160 \cdot {(2|V_{i}|)}^{\frac{k \log 4}{\log |A|}}$, by Proposition~\ref{CC}.
Let $N_{i}$ be the lower integer part of $160 \cdot {(2|V_{i}|)}^{\frac{k \log 4}{\log |A|}}$ and set $N = 1 + \max_i N_i $.
By definition, $ N \cdot A_i = V_i $ for all $ i = 1, \dots, j $.

Let $ v $ be an arbitrary vector from $ U = V_1 \oplus \dots \oplus V_j $.
It has a unique decomposition of the form $ v = v_1 + \dots + v_j $, where $ v_i \in V_i $, $ i = 1, \dots, j $.
Let $ l(v) $ denote the number of nonzero projections of $ v $, i.e.\
$$ l(v) = |\{ i \in \{ 1, \dots, j \} \mid v_i \neq 0 \}|. $$

\begin{lemma}\label{removeLarge}
	If $v \in U$, $ v \neq 0 $, then $ d(4N)^{l(v)+1} \cdot (v^{H} \cup \{ 0 \}) = V $.
\end{lemma}
\begin{proof}
	We use induction on $ l(v) $. Suppose that $ l(v) = 1 $. Then $ v $ lies in some $ V_i $
	for $ i \in \{ 1, \dots, j \} $, and hence $ N \cdot v^A = V_i $ by the definition of $N$.
	Therefore $ dN \cdot (v^{H} \cup \{0\}) = V $ by Lemma~\ref{subgroupTrick}.
	
	Suppose that $ l(v) > 1 $ and $ v_i \neq 0 $ for some $ i \in \{ 1, \dots, j \} $.
	As $ N \cdot v_i^A = V_i $, the projection of $ N \cdot (v^{A} \cup \{0\}) $ on $ V_i $ is equal to $ V_i $. 
	In particular, $ |N \cdot (v^{A} \cup \{0\})| \geq p^{f_i} $.
	
	If $ |2N \cdot (v^{A} \cup \{0\})| = p^{f_i} $, then $ 2N \cdot (v^{A} \cup \{ 0\}) = N \cdot (v^{A} \cup \{0\}) $
	and $ N \cdot (v^{A} \cup \{0\}) $ is a subgroup of $ V $. Therefore $dN \cdot (v^{H} \cup \{0\}) = V$ by Lemma~\ref{subgroupTrick}.  
	
	Assume that $ |2N \cdot (v^{A} \cup \{0\})| > p^{f_i} $. There exist two distinct vectors $ u, u' \in 2N \cdot (v^{A} \cup \{0\}) $
	with equal projections on $ V_i $, i.e.\ $ u_i = u_i' $. Since the projection of $ 2N \cdot (v^A \cup \{ 0 \}) $
	on $ V_i $ is equal to $ V_i $, there exists a vector $ w \in 2N \cdot (v^{A} \cup \{0\}) $
	with projection $ -u_i $ on $ V_i $. At least one of the vectors $ u+w $ or $ u'+w $ is nonzero; without loss of generality,
	$ u+w \neq 0 $. Since $ u+w $ has zero projection on $ V_i $, we have $ l(u + w) < l(v) $.
	As $$u + w \in 4N \cdot (v^{A} \cup \{0\}) \subseteq 4N \cdot (v^{H} \cup \{0\}),$$ the inductive hypothesis gives
	\begin{multline*}
		V = d(4N)^{l(u + w) + 1} \cdot ((u+w)^{H} \cup \{0\}) \subseteq d(4N)^{l(u+w) + 1} \cdot 4N \cdot (v^{H} \cup \{0\}) \subseteq\\
		\subseteq d(4N)^{l(v) + 1} \cdot (v^{H} \cup \{0\})
	\end{multline*}
	and the claim is proved.
\end{proof}

Assume first that $A$ is normal in $H$. In this case $j = k$ and $U = V$.
Each $V_i$ has equal size and $|V_{i}|^{k} = |V|$. Let $v$ be an arbitrary nonzero vector in $V$.
We have $l(v) \leq k \leq d$, therefore
$$\overrightarrow{\mathrm{diam}}(V,H) \leq d(4N)^{k+1} \leq d (640 \cdot {(2^{k}|V|)}^{\frac{\log 4}{\log |A|}} + 4)^{k+1} <$$
$$< d \cdot 644^{k+1} \cdot {(2^{k}|V|)}^{\frac{(k+1)\log 4}{\log |A|}} <  d \cdot 2576^{(k+1)k} \cdot {|V|}^{\frac{(k+1)\log 4}{\log |A|}}$$
by Lemma~\ref{removeLarge} and Equation~(\ref{eqDiamd}). Thus part~(2) of Theorem~\ref{mainExplicit} is proved.

Assume now that $A$ is not necessarily a normal subgroup of $H$. Let $v$ be an arbitrary nonzero vector in $V$.
By Lemma~\ref{removeSmall}, we have a nonzero vector $w$ in $U$ such that $w \in 2^{k-1} \cdot (v^{H} \cup -v^{H} \cup \{ 0\})$.
Thus
$$V =  d (4N)^{l(w)+1} \cdot (w^{H} \cup \{ 0 \}) \subseteq d (4N)^{l(w)+1} \cdot 2^{k-1} \cdot (v^{H} \cup -v^{H} \cup \{ 0\})$$
by Lemma~\ref{removeLarge}. It follows from Equation~(\ref{eqDiam}) that
\begin{multline*}
	\mathrm{diam}(V,H) \leq d \cdot (4N)^{k+1} \cdot 2^{k-1} \leq
	2d \cdot 8^{k} \cdot 161^{k+1} \cdot {(2|V|)}^{\frac{k (k+1)\log 4}{\log |A|}} <\\
	322 d \cdot 1288^{k} \cdot {(2|V|)}^{\frac{k (k+1)\log 4}{\log |A|}} <
	322 d \cdot 144^{k(k+1)} \cdot {|V|}^{\frac{k (k+1)\log 4}{\log |A|}},
\end{multline*}
and part~(1) of Theorem~\ref{mainExplicit} is proved.

\section{Proof of Theorem~\ref{main}}\label{secLP}

Recall that an element of a finite group is called a {\it $p'$-element} if its order is not divisible by the prime $p$.
The following lemma shows that finite simple groups of Lie type contain elements of large order.

\begin{lemma}\label{orderInSimp}
	Any finite simple group of Lie type in characteristic $ p $ contains a $p'$-element of order at least $p^{1/5}$.   
\end{lemma}
\begin{proof}
	Assume first that $p \geq 7$. 	
	Let $ S $ be a finite simple group of Lie type in characteristic $ p $.
	Let $ K $ be the corresponding universal version, and recall that $ S \simeq K/Z(K) $.
	By~\cite[Theorem~2.4.7~(d)]{GLS}, the Cartan subgroup of $ K $ contains a cyclic subgroup of order $ p-1 $,
	so let $ C $ be the image of that subgroup in $ S $. Clearly $ C $ is a cyclic group of
	order not divisible by $ p $, and $ |C| \geq (p-1)/|Z(K)| $.
	
	Unless $ S $ has type $ A_l $ or $ ^2 A_l $, \cite[Table~2.2]{GLS} gives us $ |Z(K)| \leq 4 $,
	which proves the claim in this case, since $ (p-1)/4 \geq p^{1/5} $. If $ S $ has type $ A_l $ or $ ^2 A_l $, then
	one can find an element of order at least $ (p-1)/2 $, see, for instance,~\cite[Corollary~3]{Buturlakin}.
	
	Assume that $p \leq 5$. The order of $S$ has at least three distinct prime divisors by Burnside's theorem.
	It follows that $S$ must have a $p'$-element of order at least $3 > 5^{1/5} \geq p^{1/5}$.  	
\end{proof}

We need the following modular analogue of Jordan's theorem on linear groups.
The original result was proved by Weisfeiler~\cite{W, Wnotes} with the use of the classification of finite simple groups,
but we use a classification-free version due to Larsen and Pink.

\begin{proposition}[{\cite[Theorem~0.2]{LarsenPink}}]
	\label{LP}
	For every $ d $ there exists a constant $ J(d) $ depending only on $ d $
	such that any subgroup $ H $ of $ \mathrm{GL}(d,p) $, $ p $ prime,
	possesses normal subgroups $ P \leq B \leq E $ such that
	\begin{enumerate}
		\item $ |H : E| \leq J(d) $.
		\item Either $ E = B $, or $ E/B $ is a direct product of finite simple groups of Lie type in characteristic $ p $.
		\item $ B/P $ is an abelian $ p' $-group.
		\item $ P $ is a (possibly trivial) $ p $-group.
	\end{enumerate}
\end{proposition}

As a consequence of the above two results, we show that irreducible linear groups contain large abelian subgroups.

\begin{corollary}\label{lieOrAbelian}
	Let $ H $ be a subgroup of $ \mathrm{GL}(d,p) $ acting irreducibly on the vector space of dimension $d$ over $\mathbb{F}_{p}$,
	where $ p $ is prime.
	\begin{enumerate}
		\item If $H$ has a composition factor isomorphic to a finite simple group of Lie type in characteristic $p$,
			then $H$ contains a $p'$-element of order at least $p^{1/5}$.
		\item If $H$ does not have a composition factor isomorphic to a finite simple group of Lie type in characteristic $p$,
			and $ |H| \geq J(d)^2 $, where $ J(d) $ is as in Proposition~\ref{LP}, 
			then $H$ contains an abelian normal $p'$-subgroup of order at least~$|H|^{1/2}$.
	\end{enumerate}
\end{corollary}
\begin{proof}
	Assume that $H$ has a composition factor $S$ isomorphic to a finite simple group of Lie type in characteristic $p$.
	The group $S$ contains a $p'$-element $x$ of order at least $p^{1/5}$. The preimage of $x$ in $H$ has order divisible
	by a number coprime to $p$ and at least $p^{1/5}$. This proves (1).
	
	Assume that $H$ does not have a composition factor isomorphic to a finite simple group of Lie type in characteristic $p$.
	We use Proposition~\ref{LP} and its notation. Since $H$ acts irreducibly on the underlying module
	of dimension $d$ over $\mathbb{F}_{p}$, the normal $p$-subgroup $ P $ of $H$ is trivial by Clifford's theorem.
	We have $E = B$ by assumption. Since $B$ is an abelian normal $p'$-subgroup,
	we are finished if $|B| \geq |H|^{1/2}$. If $|B| < |H|^{1/2}$, then $|H|^{1/2} < |H : B| \leq J(d)$,
	contradicting our assumption $ |H| \geq J(d)^2 $. This proves (2).
\end{proof}		
	
We return to the proof of Theorem~\ref{main}.
Suppose that $ H $ contains a composition factor isomorphic to a finite simple group of Lie type in characteristic $ p $.
Then Corollary~\ref{lieOrAbelian}~(1) and Theorem~\ref{mainExplicit} give us
$$ \diam(V, H) < 322 \cdot d \cdot 144^{d(d+1)} \cdot |V|^{\frac{d(d+1)\log 4}{\log |A|}}, $$
where $ A $ is an abelian $ p' $-subgroup of $ H $ with $ |A| \geq p^{1/5} $. Thus
$$ \diam(V, H) < 322 \cdot d \cdot 144^{d(d+1)} \cdot 4^{5d^2(d+1)} \leq 2^{22d^3}, $$
where the last inequality holds for all $ d \geq 2 $.

If $ H $ does not contain a composition factor isomorphic to a finite simple group of Lie type in characteristic $ p $,
and $ |H| \geq J(d)^2 $, then Corollary~\ref{lieOrAbelian}~(2) and Theorem~\ref{mainExplicit} imply
$$ \diamd(V, H) < d \cdot 2576^{d(d+1)} \cdot |V|^{\frac{(d+1) \log 4}{\log |A| }}, $$
where $ A $ is a normal abelian $ p' $-subgroup of $ H $ with $ |A| \geq |H|^{1/2} $. Hence
$$ \diamd(V, H) < d \cdot 2576^{d(d+1)} \cdot |V|^{\frac{2(d+1) \log 4}{\log |H| }} \leq 2^{18d^2} \cdot |V|^{\frac{d \log 64}{\log |H| }}, $$
where the last inequality holds for all $ d \geq 2 $.
If $ d = 1 $, then $ V = \mathbb{F}_p $ and $ H $ is a multiplicative subgroup of $ \mathbb{F}_p^\times $.
By Proposition~\ref{CC}, we have
$$ \diamd(V, H) \leq 633 \cdot |V|^{\frac{\log 4}{\log |H|}} < 2^{18} \cdot |V|^{\frac{\log 64}{\log |H|}}, $$
so the claimed inequality holds in this case as well.

\section{Proof of Corollary~\ref{ratioBounds}}\label{secCor}

The lower bound is the first inequality from Proposition~\ref{lowerBounds}.

Let $ J(d) $ be the function from Theorem~\ref{main}.
Suppose that $ |H| < J(d)^2 $. Then $ \log |V| / \log |H| \geq \frac{d \log p}{\log J(d)^2}$ and 
taking $ f(d) \geq J(d)^4 $ we have
$$ f(d)^{\frac{\log |V|}{\log |H|}} \geq (J(d)^2)^{\frac{2d \log p}{\log J(d)^2}} = p^{2d} \geq p \cdot d. $$
By Proposition~\ref{centerIneq}, the orbital diameter is always at most $ p \cdot d $, so we are done.

Now assume that $ |H| \geq J(d)^2 $. Recall that $ |H| \leq |\GL(V)| \leq p^{d^2} $,
hence $ \log |V| / \log |H| \geq 1/d $. Now taking
$$ f(d) \geq (\max \{ 2^{22d^3},  2^{18d^2} \cdot 64^d \})^d $$
and applying Theorem~\ref{main} proves the claim.

\section{Affine groups with large orbital diameters}\label{secExamples}

In this section we provide several series of groups with orbital graphs of large diameter.

Let $K$ be a nontrivial subgroup of $\mathbb{F}_{p}^{\times}$, where $ p $ is an odd prime.
Let $S$ be a transitive permutation group of degree $d$ and let $H = K \wr S$ be a wreath product
acting (linearly) imprimitively on $V = \mathbb{F}_p^d$. By~\cite[Chapter IV~\S 15, Lemma 4]{Sup},
since $p$ is odd and $K$ is nontrivial, the group $H$ acts irreducibly on $V$.

If $\{ v_1, \ldots, v_d \}$ is the standard basis of $V$, then the orbit of $v_1$ under $H$ is $v_1^H = {v}^{K}_{1} \cup \ldots \cup {v}^{K}_{d}$.
Therefore
$$ m \cdot (v_1^H \cup \{ 0 \}) \subseteq (m \cdot (v_1^K \cup \{ 0 \})) + \dots + (m \cdot (v_d^K \cup \{ 0 \})), $$
for any $ m \geq 1 $ and hence $ \diamd(V,H) \geq \diamd(\mathbb{F}_p, K) \cdot d $ by Equation~(\ref{eqDiamd}).
As $ \mathbb{F}_p^\times \cap H = K $, Proposition~\ref{centerIneq} implies that $ \diamd(V,H) \leq \diamd(\mathbb{F}_p, K) \cdot d $.
Therefore we proved

\begin{proposition}\label{wrExample}
	Let $ p $ be an odd prime, and let $ K $ be a nontrivial subgroup of $ \mathbb{F}_p^\times $.
	If $ S $ is a transitive permutation group of degree $ d $, then the (linearly) imprimitive wreath
	product $ H = K \wr S $ acts irreducibly on $ V $ and
	$ \diamd(V, H) = \diamd(\mathbb{F}_p, K) \cdot d $.
\end{proposition}

Clearly, $ \diam(V, H) = \diam(\mathbb{F}_p, K \langle -1 \rangle) \cdot d $. By taking appropriate $ K $ it
easily follows that the inequalities on undirected diameter presented in Proposition~\ref{centerIneq} are sharp.

Let $r \geq 5$ be an integer. In the second example $H$ is the alternating
group $\mathrm{Alt}(r)$ and $p$ is an odd prime not dividing $r$. Let $\{ v_{1}, \ldots , v_{r} \}$ be the basis of
the natural permutation module of $H$ over the field $\mathbb{F}_{p}$.
This module has a proper submodule $V$ which consists of vectors $v = \sum_{i=1}^{r} \lambda_{i} v_{i}$,
where each $\lambda_{i}$ is an element of $\mathbb{F}_{p}$, such that $\sum_{i=1}^{r} \lambda_{i} = 0$.
The module $ V $ is irreducible by~\cite[Lemma~5.3.4]{KL} and has dimension $ d = r-1 $ over $\mathbb{F}_{p}$.
We will prove that $ \diam(V, H) \geq (p-1)d/4 $.

Consider the $H$-orbit $\Delta = \{ \pm (v_{i} - v_{j}) \mid 1 \leq i < j \leq r \} \subseteq V$.
Let $m$ be the smallest positive integer such that $m \cdot (\Delta \cup \{ 0 \}) = V$.
Since this is a lower bound for $\mathrm{diam}(V,H)$ by Equation~(\ref{eqDiam}),
it is sufficient to show that $m \geq ((p-1)/2) \cdot d/2$.

Observe that $m = \max_{v \in V} \{ m(v) \}$ where $m(v)$ is the smallest positive integer such that $v \in m(v) \cdot (\Delta \cup \{ 0 \})$.
Denote the elements of $\mathbb{F}_{p}$ by $0, \pm 1, \ldots , \pm (p-1)/2$.
Let $v = \sum_{k=1}^{r} \lambda_{k} v_{k}$ where each $\lambda_{k}$ is in $\mathbb{F}_{p}$.
We claim that $2 m(v) \geq \sum_{k=1}^{r} |\lambda_{k}|$. This would finish the proof of the lower bound,
as one can take $ \lambda_k = (p-1)/2 $ for $ k = 1, \dots, r-1 $.

We prove the claim by induction on $t = m(v)$. This is clear for $t \leq 1$. Assume that $t \geq 2$ and that the claim is true for $t-1$. Let $v = \sum_{k=1}^{m(v)} \delta_{k}$ where each $\delta_{k} \in \Delta$. Let $\delta_{m(v)} = v_{i} - v_{j}$ for some distinct $i$ and $j$ from $\{ 1, \ldots , r \}$. Let $v' = \sum_{k=1}^{m(v)-1} \delta_{k}$. Then $2 m(v) -2 \geq 2 m(v') \geq (\sum_{k=1}^{r} |\lambda_{k}|) - |\lambda_{i}| - |\lambda_{j}| + |\lambda_{i}-1| + |\lambda_{j}+1|$. Since $1 - |\lambda_{i}| + |\lambda_{i} - 1| \geq 0$ and $1- |\lambda_{j}| + |\lambda_{j}+1| \geq 0$, the result follows.

\begin{proposition}\label{altExample}
	For every $ d \geq 4 $ and every odd prime $ p $ not dividing $ d+1 $ there
	exists a nonabelian simple group $ H $ with $ \mathrm{diam}(V,H) \geq (p-1)d/4 $.
\end{proposition}

In Corollary~\ref{ratioBounds} it was shown that there exists a function $ f(d) $ depending on $ d $ only
such that $ \diam(V, H) \leq f(d)^{\frac{\log |V|}{\log |H|}} $ for any irreducible subgroup $ H \leq \GL(V) $,
where $ V $ has dimension $ d $ over $ \mathbb{F}_p $. Proposition~\ref{wrExample} (and also Proposition~\ref{altExample})
implies that $ f(d) $ must depend on $ d $ at least linearly.
For example, let $ S $ be the symmetric group $ \mathrm{Sym}(d) $ and $ K = \langle -1 \rangle $.
Then $ |H| = |K \wr S| = 2^d \cdot d! $ and $ \log |H| \geq d \log (d/e) $. We have $ \diam(V, H) = d(p-1)/2 $ hence
$$ d(p-1)/2 \leq f(d)^{\frac{\log |V|}{\log |H|}} \leq p^{\frac{\log f(d)}{\log (d/e)}}, $$
implying $ \log f(d) / \log (d/e) \geq 1 $ for $ d \geq 3 $. Thus $ d/e \leq f(d) $, as claimed.

In our third example, let $q$ and $d$ be integers with $q \geq 2$ and $d \geq 3$.
Zsigmondy's theorem~\cite{Zig} states that, provided $(q,d) \not= (2,6)$, there is a
prime $r$ dividing $q^{d}-1$ but not dividing $q^{i}-1$ for all $i < d$.
Such a prime $r$ is called a {\it Zsigmondy prime} and is denoted by
$q_{d}$. There may be several Zsigmondy primes $q_{d}$ for given $d$ and $q$.
It is easy to see that $q_{d}$ must be congruent to $1$ modulo $d$. If $d+1$ is
prime then $d+1$ is a Zsigmondy prime $q_{d}$ whenever the multiplicative order
of $q$ modulo $d+1$ is $d$. For a given prime $d+1$ there are infinitely many
primes $q$ with this property by Dirichlet's prime number theorem. 

Let $p$ and $d$ be such that $d+1$ is an odd Zsigmondy prime $p_{d}$.  Set $ V = \mathbb{F}_p^d $
and let $h$ be an element of $\mathrm{GL}(V)$ of order $d+1$.
Since $\langle h \rangle$ is a cyclic group of order coprime to $p$, the
$\mathbb{F}_p\langle h \rangle$-module $V$ is completely reducible by Maschke's
theorem. Moreover, since the order of $\langle h \rangle$ is a Zsigmondy prime, $V$ must be irreducible. 

Let $\Delta$ be an orbit of $\langle h \rangle$. It must have size $d+1$ and it must contain a basis of~$V$;
set $ \Delta = \{ v_{1}, \ldots , v_{d}, v \}$, where $v_1, \dots, v_d$ are linearly independent.
Since $\Delta$ is an orbit, $h$ preserves the sum over all vectors in $\Delta$:
$$(v_1 + \dots + v_d + v)^{h} = v_1 + \dots + v_d + v,$$
and as $\langle h \rangle$ acts irreducibly on $V$, this sum must be zero. Therefore $v = - \sum_{j=1}^{d} v_{j}$.

Let $H$ be the cyclic group $\langle h \rangle \langle -1 \rangle$.
By the previous paragraph every orbit $\Delta$ of $H$ has the form
$\{ \pm v_{1} \} \cup \ldots \cup \{ \pm v_{d} \} \cup \{ \pm (\sum_{i=1}^{d} v_{i}) \}$ where $\{ v_{1}, \ldots , v_{d} \}$ is a basis for $V$.
Let $\ell$ be the smallest positive integer such that $\ell \cdot (\Delta \cup \{ 0 \}) = V$.
Since $ \Delta = -\Delta $, this is equal to $\mathrm{diam}(V,H)$ by Equation~(\ref{eqDiam}).
We proceed to show that $(p-1)d/4 \leq \ell \leq (p-1)(d+1)/4$, provided that $ p $ is odd.

Assume that $p$ is odd. Let $x$, $0 = x_{0}$, $x_{1}, \ldots , x_{d}$ be elements of $\mathbb{F}_p$.
For two elements $f_{1}$ and $f_{2}$ of $\mathbb{F}_p$ let $D(f_{1},f_{2})$ be their distance defined to be the smallest integer $m$
for which $f_{1} \in f_{2} + m \{ \pm 1, 0 \}$. The smallest integer $m$ for which the vector $\sum_{i=1}^{d} x_{i} v_{i}$ is contained in
$m \cdot (\Delta \cup \{ 0 \})$ is $\min_{x \in \mathbb{F}_p} \{ \sum_{i=0}^{d} D(x_{i} , x) \}$.
We have 
\begin{multline*}
	\min_{x \in \mathbb{F}_p} \{ \sum_{i=0}^{d} D(x_{i},x) \} \geq \min_{f_{d/2} \in \mathbb{F}_p} \{ D(x_{d/2}, f_{d/2}) \} +
	\sum_{i=0}^{\frac{d}{2}-1} \min_{f_{i} \in \mathbb{F}_p} \{ D(x_{i},f_{i}) + D(x_{d-i},f_{i}) \} = \\
	= \sum_{i=0}^{\frac{d}{2}-1} \min_{f_{i} \in \mathbb{F}_p} \{ D(x_{i},f_{i}) + D(x_{d-i},f_{i}) \} =
	\sum_{i=0}^{\frac{d}{2}-1} D(x_{d-i}, x_{i}).  
\end{multline*}
It follows that
$$\ell = \max_{\substack{x_{1}, \ldots , x_{d} \in \mathbb{F}_p\\ x_0 = 0}} \{ \min_{x \in \mathbb{F}_p} \{ \sum_{i=0}^{d} D(x_{i}, x) \} \} \geq
 \max_{\substack{x_{1}, \ldots , x_{d} \in \mathbb{F}_p\\ x_0 = 0}} \sum_{i=0}^{(d/2)-1} D(x_{d-i}, x_{i}) \geq \frac{(p-1)d}{4}.$$ 
On the other hand,
$$\min_{x \in \mathbb{F}_p} \{ \sum_{i=0}^{d} D(x_{i} , x) \} \leq \min \{ \sum_{i=0}^{d} D(x_{i} , 0)  , \sum_{i=0}^{d} D(x_{i} , (p-1)/2) \}.$$ 
Without loss of generality we can renumber $ x_0, \dots, x_d $ in such a way that for some $ z $ and $ k $
we have $x_{0} = \ldots = x_{z} = 0$, $1 \leq x_{z+1} \leq \ldots \leq x_{k} \leq (p-1)/2$ and
$-1 \geq x_{k+1} \geq \ldots \geq x_{d} \geq -(p-1)/2$. With this assumption,
$$\sum_{i=0}^{d} D(x_{i} , 0) = ( \sum_{i=z+1}^{k} x_{i} ) - ( \sum_{i=k+1}^{d} x_{i} ),$$
$$\sum_{i=0}^{d} D(x_{i} , (p-1)/2) = \frac{(p-1)(d+1)}{2}  - ( \sum_{i=z+1}^{k} x_{i} ) +( \sum_{i=k+1}^{d} x_{i} ).$$
It follows that
$$\ell = \max_{\substack{x_{1}, \ldots , x_{d} \in \mathbb{F}_p\\ x_0 = 0}} \{ \min_{x \in \mathbb{F}_p} \{ \sum_{i=0}^{d} D(x_{i}, x) \} \} \leq \frac{(p-1)(d+1)}{4}.$$ 
We summarize the discussed example in the following

\begin{proposition}
	There are infinitely many positive integers $ d $ and, for each such $ d $, infinitely many primes $ p $ and
	cyclic groups $ H $ for which $ (p-1)d/4 \leq \mathrm{diam}(V, H) \leq (p-1)(d+1)/4 $.
\end{proposition}

\bigskip

\centerline{\bf Acknowledgement}

The authors would like to thank Sean Eberhard and L\'aszl\'o Pyber for comments on an earlier version of the paper.


\begin{thebibliography}{30}
	\bibitem{AB} F. Annexstein, M. Baumslag, On the diameter and bisector size of Cayley graphs.
		\emph{Math. Systems Theory} \textbf{26} (1993), no. 3, 271--291. 
	
	\bibitem{many} L. Babai, G. Hetyei, W. M. Kantor, A. Lubotzky, \'A. Seress, On the diameter of finite groups. 31st Annual Symposium on Foundations of Computer Science, Vol. I, II (St. Louis, MO, 1990), 857--865, IEEE Comput. Soc. Press, Los Alamitos, CA, 1990.
	
	\bibitem{BS} L. Babai, \'A. Seress, On the diameter of permutation groups. \emph{European J. Combin.} \textbf{13} (1992), no. 4, 231--243.
	
	\bibitem{Buturlakin} A. A. Buturlakin, Spectra of finite linear and unitary groups. \emph{Algebra and Logic.} \textbf{47} (2008), no. 2, 91--99.
	\bibitem{Cameron} P. J. Cameron, Permutation groups, Cambridge University Press, 1999.
		
	\bibitem{CPSS} P. J. Cameron, C. E. Praeger, J. Saxl, G. M. Seitz, On the Sims conjecture and distance transitive graphs,
		\emph{Bull. London Math. Soc.} \textbf{15} (1983), 499--506.
	
	\bibitem{CC} T. Cochrane, J. Cipra, Sum-product estimates applied to Waring's problem over finite fields. \emph{Integers} \textbf{12} (2012), no. 3, 385--403.

	\bibitem{Collins} M. J. Collins, Modular analogues of Jordan's theorem for finite linear groups.
		\emph{J. Reine Angew. Math} \textbf{2008} (2008), no. 624, 143--171.
	
	\bibitem{GLS} D.~Gorenstein, R.~Lyons, R.~Solomon, The classification of the finite simple groups, Number 3. Mathematical surveys and monographs, American Mathematical Society, 1994.
	
	\bibitem{Higman} D. G. Higman, Intersection matrices for finite permutation groups.
	\emph{J. Algebra} \textbf{6} (1967) 22--42.

	\bibitem{KL} P. Kleidman, M. Liebeck, The subgroup structure of the finite classical groups.
		London Mathematical Society Lecture Note Series, 129. Cambridge University Press, Cambridge, 1990.
	
	\bibitem{LarsenPink} M. J. Larsen, R. Pink, Finite subgroups of algebraic groups. \emph{J. Amer. Math. Soc.} \textbf{24} (2011), no. 4, 1105--1158.
	\bibitem{LL} R. Lawther, M. W. Liebeck, On the diameter of a Cayley graph of a simple group of Lie type based on a conjugacy class.
		\emph{J. Combin. Theory Ser. A} \textbf{83} (1998), no. 1, 118--137.
	
	\bibitem{LMT} M. W. Liebeck, D. Macpherson, K. Tent, Primitive permutation groups of bounded orbital diameter.
	\emph{Proc. Lond. Math. Soc.} (3) \textbf{100} (2010), no. 1, 216--248.
	
	\bibitem{P} S. Peluse, On exponential sums over orbits in $\mathbb{F}^{d}_{p}$. ArXiv:1606.03495.

	\bibitem{Pyb} L. Pyber, A bound for the diameter of distance-regular graphs. \emph{Combinatorica} \textbf{19} (4) (1999), 549--553.

	\bibitem{PT} L. Pyber, G. Tracey, Some simplifications in the proof of the Sims conjecture. ArXiv:2102.06670.

	\bibitem{Re} K. Rekv\'enyi, On the orbital diameter of groups of diagonal type. ArXiv:2102.09867. 

	\bibitem{Sh} A. Sheikh, Orbital diameters of the symmetric and alternating groups.
		\emph{J. Algebraic Combin.} \textbf{45} (2017), no. 1, 1--32. 

	\bibitem{Sup} D. A. Suprunenko, Matrix groups. \emph{Amer. Math. Soc.} \textbf{45}, 1976.

	\bibitem{W} B. Weisfeiler, Post-classification version of Jordan’s theorem on finite linear groups.
		\emph{Proc. Natl. Acad. Sci. USA} \textbf{81} (1984), 5278--5279. 
		
	\bibitem{Wnotes} B. Weisfeiler, On the size and structure of finite linear groups. ArXiv:1203.1960.

	\bibitem{Zig} K. Zsigmondy, Zur Theorie der Potenzreste. \emph{Monatsh. für Math. u. Phys.} \textbf{3} (1892), 265--284.
	
\end{thebibliography}
\end{document}